\documentclass{article}
\usepackage[utf8]{inputenc}
\usepackage{amsmath, amsfonts, amssymb, amsthm, framed}
\usepackage[top=1in, bottom=1in]{geometry}

\newtheorem{theorem}{Theorem}
\newtheorem{lemma}[theorem]{Lemma}

\newtheorem{corollary}[theorem]{Corollary}
\newtheorem{definition}[theorem]{Definition}

\newtheorem{remark}[theorem]{Remark}

\newcommand{\tbf}[1]{\textbf{#1}}
\newcommand{\scr}[1]{\mathcal{#1}}

\usepackage{color}

\newcommand{\eps}{\varepsilon}

\newcommand{\ignore}[1]{}

\newcommand{\er}[3]{\mathcal{G}^{(#2)}_{#1 , #3} }
\newcommand{\kout}[3]{\mathcal{G}^{#2} _{#1} (#3-\text{out})}

\title{\vspace*{-.8in}Hamiltonian Berge Cycles in Random Hypergraphs\footnote{An earlier arxiv draft of this paper did not have our stopping-time results.}}
\author{Deepak Bal\footnote{Dept.\ of Mathematical Sciences, Montclair State University, Montclair NJ, USA \qquad \texttt{deepak.bal@montclair.edu}}  \qquad Ross Berkowitz\footnote{Dept.\ of Mathematics, Yale University, New Haven CT, USA \qquad \texttt{ross.berkowitz@yale.edu}} \qquad Pat Devlin\footnote{Dept.\ of Mathematics, Yale University, New Haven CT, USA \qquad \texttt{patrick.devlin@yale.edu}}  \qquad Mathias Schacht\footnote{Dept.\ of Mathematics, Yale University, New Haven CT, USA \qquad \texttt{mathias.schacht@yale.edu}}}
\date{}
\begin{document}
\renewcommand{\thefootnote}{\fnsymbol{footnote}}
\maketitle

\begin{abstract}In this note, we study the emergence of Hamiltonian Berge cycles in random $r$-uniform hypergraphs.  For $r\geq 3$, we prove an optimal stopping-time result that if edges are sequently added to an initially empty $r$-graph, then as soon as the minimum degree is at least 2, the hypergraph almost surely has such a cycle.  In particular, this determines the threshold probability for Berge Hamiltonicity of the Erd\H{o}s--R\'enyi random $r$-graph, and we also show that the $2$-out random $r$-graph almost surely has such a cycle.  We obtain similar results for \textit{weak Berge} cycles as well, thus resolving a conjecture of Poole.
\end{abstract}

\section{Introduction}

\renewcommand*{\thefootnote}{\arabic{footnote}}
An $r$-graph (or an $r$-uniform hypergraph) on $V$ is a collection of $r$-element subsets (i.e., `edges') of $V$ (the set of `vertices').  A \textit{Berge cycle} in a hypergraph is an alternating sequence of distinct vertices and edges $(v_1, e_1, \ldots , v_n, e_n)$ where $v_{i}, v_{i+1}$ are in $e_i$ for each $i$ (indices considered modulo $n$), and a \textit{Hamiltonian Berge cycle} is one in which every vertex appears.  The Erd\H{o}s--R\'enyi random $r$-graph, denoted $\er{n}{r}{p}$, is the distribution over $r$-graphs on $\{1, 2, \ldots, n\}$ in which each edge appears independently with probability $p$.

The case $r=2$ (i.e., graphs) has received particular attention.  In this setting, Hamiltonian Berge cycles are unambiguously referred to simply as Hamiltonian cycles and the question of when a random graph is likely to contain a Hamiltonian cycle is extremely well-understood \cite{KSHamiltonian, bollobasHamiltonian, AKSHamiltonian, bollobasBook}.   Historically, Berge cycles were the first among several natural generalizations of the notion of cycle from graphs to hypergraphs \cite{berge}.  
Many of these differing notions of hypergraph cycles (e.g. loose, tight, offset, etc) have been studied in the context of random $r$-graphs, with particular emphasis on determining for which parameters $\er{n}{r}{p}$ is likely to contain such a ``Hamiltonian cycle" (see \cite{kuhn2014} for a survey and \cite{DF11, DF13, DFLS12, dudek2018, parczyk2016} for examples). Of particular relevance for us, Poole \cite{poole} focused on \textit{weak Hamiltonian Berge cycles}---which are defined as Hamiltonian Berge cycles without the restriction that the edges be distinct---and for these weaker stuctures he obtained the following sharp result.
\begin{theorem}[Poole \cite{poole}] \label{thm: weak}
Suppose $r \geq 3$ is fixed, and $p = (r-1)! \dfrac{\log n  + c_n}{n^{r-1}}$.  Then we have
\[
\lim_{n \to \infty} \mathbb{P} \left( \er{n}{r}{p} \text{ has a weak Hamiltonian Berge cycle} \right) = \begin{cases}
0, \qquad &\text{if $c_n \to -\infty$}\\
e^{-e^{-c}}, \qquad &\text{if $c_n \to c \in \mathbb{R}$}\\
1, \qquad &\text{if $c_n \to \infty$.}
\end{cases}
\]
\end{theorem}
\noindent Here, as in the case of graphs, the choice of $p$ is driven by the need to avoid isolated vertices (i.e., vertices not contained in any edges), whereas for (non-weak) Hamiltonian Berge cycles, we need each vertex to have degree at least $2$.

In this note, we prove these minimum degree requirements are the primary obstructions to Hamiltonicity by showing the following \textit{stopping-time} result.  We say an event happens \textit{almost surely} to mean it occurs with probability tending to $1$ as $n$ goes to infinity.  Consider the ordinary random graph process, where at each step, a uniformly random non-edge is added to the graph. Ajtai, Koml\'os, and Szemer\'edi \cite{AKSHamiltonian} and Bollob\'as \cite{bollobasHamiltonian} proved that almost surely, the graph becomes Hamiltonian at the very same step when the minimum degree becomes two. In other words, the graph becomes Hamiltonian as soon as the ``obvious obstruction'' to Hamiltonicity disappears. 
Our main result is a generalization of this result to random $r$-graphs for the notion of Berge (and weak Berge) Hamiltonicity.

\begin{theorem}\label{stopping}
Suppose $r \geq 3$ is fixed and let $\{e_1, e_2, \ldots, e_{\binom nr} \}$ denote a random ordering of the $r$-subsets of $[n]$.  Let $\mathcal{H} (t)$ denote the $r$-graph on $[n]$ with edge set $\{e_i \ : \ 1 \leq i \leq t\}$, and let $T_{k}$ denote the minimum $t$ such that every vertex of $\mathcal{H}(t)$ is contained in at least $k$ edges.  Then (i) $\mathcal{H} (T_1)$ almost surely has a weak Hamiltonian Berge  cycle, and (ii) $\mathcal{H}(T_2)$ almost surely has a Hamiltonian Berge  cycle.
\end{theorem}

\noindent The statement that $\mathcal{H}(T_1)$ has a weak Hamiltonian Berge cycle resolves a conjecture by Poole \cite{poole}. Standard techniques also immediately imply both Theorem \ref{thm: weak} and the following corollary.

\begin{corollary}\label{er berge}
Suppose $r \geq 3$ is fixed, and $p = (r-1)! \dfrac{\log n  + \log\log n + c_n}{n^{r-1}}$.  Then we have
\[
\lim_{n \to \infty} \mathbb{P} \left( \er{n}{r}{p} \text{ has a Hamiltonian Berge cycle} \right) = \begin{cases}
0, \qquad &\text{if $c_n \to -\infty$}\\
e^{-e^{-c}}, \qquad &\text{if $c_n \to c \in \mathbb{R}$}\\
1, \qquad &\text{if $c_n \to \infty$.}
\end{cases}
\]
\end{corollary}

Previously, Clemens, Ehrenm\"uller, and Person \cite{clemens}, proved a general resilience result implying a version of Corollary \ref{er berge} with $p = \log^{k(r)} (n) / n^{r-1}$, where $k(r)$ is a constant depending on $r$. 
Our proof of Theorem \ref{stopping} follows closely a presentation of Krivelevich for the stopping time result for ordinary random graphs. 

In addition to uniform random $r$-graphs, we also study Berge Hamiltonicity of another random $r$-graph model.  The \emph{$k$-out random $r$-graph} on $V = [n]$, denoted $\kout{n}{r}{k}$, has the following distribution:
 for each $v \in V$, independently choose $k$ edges $E_v = \{e_1, e_2, \ldots, e_k\}$, where each $e_i \subseteq V$ is chosen uniformly at random from among all $r$-element sets containing $v$.  The hypergraph then consists of all edges chosen in this way: namely, $\bigcup_{v} E_v$.
We will show that it will not matter for our results if these edges are chosen with replacement or not (as is typical with this model when $r=2$), and we will adopt whichever convention suits us.

In the graph case, Hamiltonicity of this model was first studied by Fenner and Frieze \cite{fenner} who showed $\kout{n}{2}{23}$ is almost surely Hamiltonian.  This was improved incrementally by a series of authors until Bohman and Frieze \cite{bohman2009} showed that $\kout{n}{2}{3}$ is almost surely Hamiltonian (whereas $\kout{n}{2}{2}$ almost surely is not).  The generalization of the $k$-out model to hypergraphs, though natural, is not yet well-studied, and in fact the only publication we are aware of is \cite{kahn}, which addresses perfect fractional matchings.

For the $k$-out model, we settle the issue of ordinary and weak Berge Hamiltonicity completely.

\begin{theorem}\label{thm: k-out-berge-main}

For any fixed $r \geq 4$, $\kout{n}{r}{2}$ almost surely has a Hamiltonian Berge cycle.  $\kout{n}{r}{1}$ almost surely does not have a Hamiltonian Berge cycle but does have a weak Hamiltonian Berge cycle.
$\kout{n}{3}{2}$ almost surely has a  Hamiltonian Berge cycle, whereas $\kout{n}{3}{1}$ almost surely does not have a weak Hamiltonian Berge cycle.
\end{theorem}

In Section \ref{sec:stopping} we prove that  $\mathcal{H}(T_2)$ almost surely has a Hamiltonian Berge cycle (Theorem \ref{stopping} (ii)). In Section \ref{sec:stopping-weak} we sketch a proof that $\mathcal{H}(T_1)$  almost surely contains a  weak Hamiltonian Berge cycle (Theorem \ref{stopping} (i)). In Section \ref{sec:k-out} we prove Theorem \ref{thm: k-out-berge-main}.
Throughout, all logarithms are natural.


%

\section{Stopping time result for Berge Hamiltonicity}\label{sec:stopping}
Our proof is very close to the proof of the stopping time result for Hamiltonicity of ordinary random graphs as presented by Krivelevich in \cite{krivelevich}. We use the famous  P\'osa extension-rotation technique and the concept of boosters. We start with a few definitions.

\begin{definition}
A hypergraph is a \emph{$(k, \alpha)$-expander} iff for all disjoint sets of vertices $X$ and $Y$, if $|Y| < \alpha |X|$ and $|X| \leq k$, then there is an edge, $e$, such that $|e \cap X| = 1$ and $e \cap Y = \emptyset$.
\end{definition}

\begin{definition}\label{def:booster}
For a hypergraph $G$, a \emph{booster} is an edge such that either $G \cup e$ has a longer (Berge) path than $G$ or $G \cup e$ is (Berge) Hamiltonian.
\end{definition}

%
%
%
%

\subsection{Statements of Lemmas}\label{sec:lemmas}

The lemmas of this section can be summarized as follows.
\begin{itemize}
\item[(i)] Non-Hamiltonian expansive hypergraphs have lots of boosters (P\'osa rotations, Lemma \ref{boosters in expanders})
\item[(ii)] $\mathcal{H}(T_2)$ almost surely has a booster for each sparse expansive sub-hypergraphs (Lemma \ref{stopping contains boosters})
\item[(iii)] $\mathcal{H}(T_2)$ almost surely contains a sparse expansive sub-hypergraph (Lemmas \ref{random has nice conditions}, \ref{theorem conditions imply expansive})
\end{itemize}

For the formal statements, we need a bit of notation.  For any $r$-graph $G$, let
\[SMALL(G) := \{v \ : \ d(v) \leq \varepsilon \log(n)\}\] 
for $\varepsilon > 0$ small, to be determined. We also define a random subgraph $\Gamma_0\subset G$ as follows. Every vertex $v\not\in SMALL(G)$ chooses a subset $E_v$ of $\eps \log n$ many edges uniformly at random from the set of all edges incident to $v$. For every $v\in SMALL(G)$, let $E_v$ be the set of all edges incident to $v$. Then the edge set  of $\Gamma_0$ is defined as $E(\Gamma_0):=\bigcup_{v}E_v.$

\begin{lemma}\label{boosters in expanders}
There exists a constant $c_r > 0$ such that if $G$ is a connected $(k,2)$-expander $r$-graph on at least $r+1$ vertices, then $G$ is Hamiltonian, or it has at least $k^2 n^{r-2} c_r$ boosters.
\end{lemma}

\begin{lemma}\label{stopping contains boosters}
Let $G = \mathcal{H}(T_2)$.  Then with high probability, if $\Gamma \subseteq G$ is \textbf{any} $(n/4, 2)$-expander with $|E(\Gamma)| \leq \varepsilon \log(n) n + n$, then $\Gamma$ is Hamiltonian or $G$ has at least one booster edge of $\Gamma$.
\end{lemma}

\begin{lemma}\label{random has nice conditions}
Let $G= \mathcal{H}(T_2)$.  Then almost surely, $G$ has the following properties:
\begin{itemize}
\item[(P1)] $\Delta(G) \leq 10 \log(n)$ 
\item[(P2)] $|SMALL(G)| \leq n^{.9}$
\item[(P3)] Let $N = \{v \in [n] \ : \ \exists e\in E(G), \ v \in e,\ SMALL(G) \cap e \neq \emptyset\}$.  No edge meets $SMALL(G)$ more than once, and no $u \notin SMALL(G)$ lies in more than one edge meeting $N \setminus \{u\}$.
\item[(P4)] If $U \subseteq [n]$ has size at most $|U| \leq \dfrac{n}{\log(n)^{1/2}}$, then there are at most $|U| \log(n) ^{3/4}$ edges of $G$ that meet $U$ more than once
\item[(P5)] for every pair of disjoint vertex sets $U,W$ of sizes $|U| \leq \dfrac{n}{\log(n)^{1/2}}$ and $|W| \leq |U| \log(n) ^{1/4}$, there are at most $\dfrac{\varepsilon \log(n) |U|}{2}$ edges of $G$ meeting $U$ exactly once and also meeting $W$
\item[(P6)] for every pair of disjoint vertex sets $U,W$ of sizes $|U| = \dfrac{n}{\log(n) ^{1/2}}$ and $|W| =n/4$, there are at least $n \log(n)^{1/3}$ edges of $G$ meeting $U$ exactly once and $W$ exactly $r-1$ times.
\end{itemize}
Almost surely (over the choices of $E_v$), $\Gamma_0$ also has the property
\begin{itemize}
\item[(P7)] for every pair of disjoint vertex sets $U,W$ of sizes $|U| = \dfrac{n}{\log(n) ^{1/2}}$ and $|W| =n/4$, there is at least one edge in $\Gamma_0$ meeting $U$ exactly once and $W$ exactly $r-1$ times.
\end{itemize}
\end{lemma}

\begin{lemma}\label{theorem conditions imply expansive}
Deterministically, if $\Gamma_0 \subset G$ satisfies $\delta(\Gamma_0) \geq 2$ and (P3), (P4), (P5), and (P7), then $\Gamma_0$ is a connected $(n/4, 2)$-expander.  
\end{lemma}

\subsection{Why we're done modulo proofs of the above}\label{sec:pf-main-ii}
\begin{proof}[Proof of Theorem \ref{stopping} (ii)]
Let $G = \mathcal{H}(T_2)$ and let $\Gamma_0\subset G$ as in the statement of Lemma \ref{random has nice conditions}. Then by definition, $|E
(\Gamma_0)| \le \eps n \log n$ and by Lemma \ref{theorem conditions imply expansive}, $\Gamma_0$ is a connected $(n/4, 2)$-expander.
Now we start with $\Gamma_0$ and iteratively add boosters until we arrive at a Hamiltonian hypergraph. Clearly this cannot be repeated more than $n$ times as the length of the longest path increases at each step. Also since at each step we have an $(n/4, 2)$-expander with at most $\eps n \log n + n$ many edges, Lemma \ref{stopping contains boosters} guarantees the existence of a Hamiltonian cycle or a booster to add.
\end{proof}

%

\subsection{Proofs of Lemmas}

\begin{proof}[Proof of  Lemma \ref{boosters in expanders}]
Say $G$ is a $(k,2)$-expander, and suppose it is not Hamiltonian.  Let 
$P = v_1, v_2, \ldots, v_m$ be any longest path in $G$.  Suppose $e$ is an edge containing $v_m$.

\textbf{Case I:} suppose $e$ is not involved in the path.  Then $e$ cannot contain any vertices outside of $P$ or else we could add that to get a longer path.  Say $v_m \neq v_{j} \in e$.  Then we can add $e$ to our path and delete the edge $v_{j} \sim v_{j+1}$ to obtain a new path $v_1, v_2, \ldots , v_j, v_m, v_{m-1}, v_{m-2}, \ldots, v_{j+1}$.  \textit{Such a move is called a rotation.}

\textbf{Case II:} suppose $e$ \textit{is} involved in the path, and say $e$ is needed to connected $v_i$ to $v_{i+1}$.  Then we can replace this path via another rotation $v_1, v_2, \ldots, v_i , v_m, v_{m-1}, \ldots, v_{i+1}$.  (If $v_{i+1} = v_m$ then this rotation actually didn't do anything.)

For fixed vertex $v_1$ and initial path $P$, let $R(v_1)$ be the vertices that could possibly appear as right endpoints starting with $P$ and doing rotaitons.  Let $R^{\pm} = \{v_i : \{v_{i-1}, v_{i+1}\} \cap R \neq \emptyset \}$ (with vertices numbered as in initial path).  Then notice that edges in case I intersect $R^{\pm}$ in (at least) $r-1$ points since for every $v_m \neq v_i \in e$, we can rotate so that $v_{i+1}$ is at the end.  And if $e$ is as in case II, then $v_i \in R^{\pm}$.  Finally (just as in P\'{o}sa's original argument), the set $R^{\pm}$ is the same for the orginal path as it is for any rotation.

If $e$ is an edge containing some $x \in R$ then $e$ must meet $R^{\pm}$ in at least one vertex \textit{other than $x$}.  Therefore, if $e$ meets $R$ in exactly one point, then it must meet $R^{\pm} \setminus R$.  Furthermore, $|R^{\pm} \setminus R| \leq |R^{\pm}| < 2 |R|$ (with strict inequality since every element of $R$ has at most $2$ neighbors except for the rightmost, which has only $1$).  And since $G$ is expansive, $|R| \geq k$.

So for each vertex that can be chosen as a left endpoint of a longest path, there are at least $k$ right endpoints we can have.  And by reversing the roles of left and right, we have at least $k$ left endpoints possible.   Suppose $(u,v)$ is a pair with these vertices as endpoints of a longest path, say $P$.  If $e$ is an edge of the hypergraph not contained in $P$, then we cannot have $\{u,v\} \subseteq e$.  Otherwise, if $P$ already contains all the vertices this would be a Hamiltonian cycle.  And if $P$ does not contain all the vertices, let $x$ be a vertex not on $P$.  Then since the graph is connected, there is a path from $x$ to the cycle $P + e$.  The last step of this path must be of the form $u \sim v_{j}$ for some vertex $v_j$ and some $u$ not in the path.  But then we can get a longer path by including this edge and $u$ (and deleting at most one edge of $P+e$ to use when connecting $u$ to this cycle).

Thus, for each $(u,v)$, there are at least $\binom{n-2}{r-2} - (n-1)$ \textit{booster edges} containing $u$ and $v$ (where the ``$-(n-1)$" is to avoid counting any edges already contained in the path).  Summing over all choices of $(u,v)$ [at least $k^2$] we have
\[
r \cdot (r-1) \cdot \# (boosters) \geq k^2 \left[ \binom{n-2}{r-2}- (n-1) \right].
\]
In the case $r=2$, we can ignore the $(n-1)$ term completely since if $u \sim v$, then the longest path contains only 2 vertices (not possible in a connected graph on $3$ vertices).  And in the case $r=3$, then $(n-1)$ term can be replaced by $1$ since there is at most one edge used in the path that also contains $\{u,v\}$.  In either event, we obtain $\# (boosters) \geq k^2 n^{r-2} c_r$ for some constant $c_r > 0$.
\end{proof}

\hrulefill

The proofs of Lemmas \ref{stopping contains boosters} and \ref{random has nice conditions} are very similar to those which appear in Krivelevich \cite{krivelevich},  Alon-Krivelevich \cite{alon-krivelevich} and Devlin-Kahn \cite{kahn}. Thus we have deferred their proofs to the Appendix.

\hrulefill

\begin{proof}[Proof of Lemma \ref{theorem conditions imply expansive}]
Let $S$ be a subset of $[n]$, and say $S_1 = S \cap SMALL(G)$ and $S_2 = S \setminus S_1$.

\tbf{Case I:} 
Suppose $n/4 \geq |S| \geq n / \log(n) ^{1/2}$.  Let $Y$ be a set disjoint from $S$ such that $Y$ covers $S$ (i.e., every edge meeting $S$ exactly once also meets $Y$) and $|Y| < 2 |S|$.  Then let $W = [n] \setminus (S \cup Y)$, then (because $|S| \leq n/4$), we have $|W| \geq n/4$.  But (P7) implies [after fist making $S$ and $W$ smaller as needed] that there's an edge meeting $S$ exactly once and $W$ in $r-1$ spots,  a contradiction.

\tbf{Case II:} Suppose $|S| \leq n / \log(n)^{1/2}$.  Suppose $Y$ is a set disjoint from $S$ such that $Y$ covers $S$ and $|Y| < 2|S|$.  Say $Y_1 = Y \cap N(SMALL)$ (i.e., each vertex of $Y_1$ is adjacent to something in $S_1$), and let $Y_2 = Y \setminus Y_1$.

Then $Y_1 \cup S_2$ covers $S_1$ and $Y \cup S_1$ covers $S_2$ 
Because $Y_1 \cup S_2$ covers $S_1$, we have
\[
|Y_1 \cup S_2| \geq 2 |S_1|
\]
because the edges of $S_1$ are sufficiently spread out by (P3), and each vertex is on at least 2 edges.

Now by (P4) there are at least $|S_2| (\varepsilon \log(n) - \log(n)^{3/4})$ edges that intersect $S_2$ exactly once.  And for each $u \in S_2$, there is at most one edge through $u$ meeting $S_1 \cup Y_1$ by (P3).  Therefore, there are at least $|S_2| (\varepsilon \log(n) - \log(n)^{3/4} -1)$ edges meeting $S_2$ exactly once and not meeting $S_1 \cup Y_1$ at all.  So there are at least $|S_2| (\varepsilon \log(n) - \log(n)^{3/4} -1) > |S_2| \varepsilon \log(n) /2$ edges that hit $S_2$ exactly once and then also hit $Y_2$.  Therefore, by (P5) we have $|Y_2| \geq |S_2| \log(n) ^{1/4}$.
So in total, we have
\[
|Y| = |Y_1| + |Y_2| \geq |Y_1 \cup S_2| - |S_2| + |Y_2| \geq 2 |S_1| -|S_2| + |S_2| \log(n) ^{1/4} \geq 2 |S_1| + 2|S_2|
\]
again, a contradiction thereby completing Case II.

Finally, to see that  $\Gamma_0$ is connected, note that $(n/4, 2)$-expansive implies that $\Gamma_0$ has no connected component of size less than $n/4$. But then (P7) implies that any disjoint sets of size at least $n/4$ have an edge between them.
\end{proof}

\section{Weak Berge Hamiltonicity}\label{sec:stopping-weak}
In this section we prove Theorem \ref{stopping} (i), i.e., that $\scr{H}(T_1)$ almost surely contains a weak Hamiltonian Berge cycle. This resolves a conjecture of Poole from \cite{poole}.
The proof is almost the same as in the previous section (and in fact, we can reuse most of the previous results). In this section we sketch the proof, pointing out what changes when dealing with weak Hamiltonicity.

\begin{proof}
[Proof sketch of Theorem \ref{stopping} (i)]
\begin{definition}
A hypergraph is a \emph{weak $(k,\alpha)$-expander} iff the following happens.  If $X, Y$ are disjoint subsets of vertices, $|Y| < \alpha |X|$, and every edge meeting $X$ is contained in $X \cup Y$, then $|X| \geq k$.
\end{definition}

\begin{remark}
We use the word ``weak'' here only to refer to weak Hamiltonicity. The notions of \textit{weak expansive} and \textit{expansive} are incomparable.  weak-$(k,\alpha)$-expansive means for all $|X| \leq k$, we have $\alpha |X| \le |N(X) \setminus X|$. 
\end{remark}
In this section, the notion of ``booster'' now refers to an edge whose addition increases the length of the longest \emph{weak} Berge path or introduces a \emph{weak} Hamiltonian Berge cycle.
The corresponding Lemmas in Section \ref{sec:lemmas} and their proofs are virtually the same except for the following slight changes.
\begin{itemize}
\item \tbf{Lemma \ref{boosters in expanders}}: Use the weak notions of $(k,2)$-expander
 and Hamiltonicity. For the proof, notice that Case II of the proof of Lemma \ref{boosters in expanders} doesn't matter (we can reuse edges even if they're already in the path).  So we see that every edge meeting $R$ at some point $v$ must be contained in $\{v\} \cup R^{\pm}$.  Thus, each edge meeting $R$ is contained in $R \cup (R^{\pm} \setminus R)$.  We also know that $|R^{\pm} \setminus R| \leq |R^{\pm}| < 2 |R|$, so by weak-expansion, we know $|R| \geq k$.  The rest of the proof proceeds as before.

\item \tbf{Lemmas \ref{stopping contains boosters} and \ref{random has nice conditions}}: In the statements, use $G = \mathcal{H}(T_1)$ and the weak notions of expansion and Hamiltonicity. The proofs remain unchanged.
 
 
 \item \tbf{Lemma \ref{theorem conditions imply expansive}}: For the statement, suppose $\delta(\Gamma_0) \ge 1$ and the 4 conditions and conclude weak expansion. The proof is exactly the same except for the statement ``$|Y_1 \cup S_2| \geq 2|S_1|$."  In this case, we know that every edge meeting $S_1$ is contained in $Y_1 \cup S_2 \cup S_1$.  By (P3), we also know that every edge meeting $S_1$ intersects it exactly once and also that any two edges meeting $S_1$ do not intersect outside of $S_1$.  And since $\delta(\Gamma_0) \geq 1$, there are at least $|S_1|$ edges meeting $S_1$, and (by (P3)) the union of these edges is at least at least $(r-1)|S_1|$ vertices outside of $S_1$.  This gives us $|Y_1 \cup S_2| \geq (r-1) |S_1| \geq 2 |S_1|$ (since $r \geq 3$), which is stronger than what is needed anyway.  The rest of the proof is identical.

In fact, this proof shows $\Gamma_0$ satisfying the assumptions is a \textit{weak}-$(n/4, r-1)$-expander.  (The idea being that there's a perfect matching covering SMALL($\Gamma_0$), and the rest of the graph is extremely expansive.)
 \end{itemize}
With these adapted Lemmas, we can finish the proof of Theorem \ref{stopping} (i) in exactly the same fashion as the proof of Theorem \ref{stopping} (ii) in Section \ref{sec:pf-main-ii}.
\end{proof}

\section{$k$-out model}\label{sec:k-out}

Before proving Theorem \ref{thm: k-out-berge-main},
we prove the following, which justifies our claim that it does not matter whether the edges in $\kout{n}{r}{k}$ are chosen with replacement (since almost surely the selected edges are distinct).
\begin{lemma}\label{distinct edges}
For any fixed $k$ and $r \geq 3$, $\kout{n}{r}{k}$ almost surely has exactly $nk$ edges.
\end{lemma}
\begin{proof}
Suppose the edges chosen to form $\kout{n}{r}{k}$ are labelled as $e_v ^{(j)}$ where $v \in V$ and $j \in \{1, 2, \ldots , k\}$ so that $E_v = \{e_v ^{(j)} : j\}$.  If $v \neq u$ and $i, j$ are fixed, then $\mathbb{P} \left( e_v ^{(i)} = e_u ^{(j)} \right)$ is at most ${n-2 \choose r-2} / {n-1 \choose r-1}^2 \sim c_r n^{-r}$, for some\footnote{Here we use the convention that for functions $f$ and $g$, we write $f \sim g$ to mean $\lim_{n \to \infty} f / g = 1$.} constant $c_r$ depending only on $r$.  Therefore, the probability that there exist edges $e_{v} ^{(i)} = e_{u} ^{(j)}$ with $v \neq u$ is asymptotically at most $(nk)^2 c_r n^{-r}$, which tends to $0$ since $k$ is fixed and $r > 2$.  Moreover, if we are considering the model where $e_v ^{(1)}, e_v ^{(2)}, \ldots$ are selected with replacement, then the probability that there are two of these that are equal is at most $k^2 n {n-1 \choose r-1} / {n-1 \choose r-1}^2 \sim k^2 c_r ' n^{-r+2}$, which also tends to $0$.  Thus, regardless of whether or not we select the edges of $\kout{n}{r}{k}$ with replacement, when $r \geq 3$, the $r$-graph almost surely has $nk$ edges.
\end{proof}

First we handle the case of (ordinary) Berge Hamiltonicity.
\begin{theorem}\label{thm: k-out-berge}
For any fixed $r \geq 3$, $\kout{n}{r}{2}$ almost surely has a Hamiltonian Berge cycle, whereas $\kout{n}{r}{1}$ almost surely does not.
\end{theorem}

\begin{proof}[Proof of Theorem \ref{thm: k-out-berge}.]
First we will show that for $r \geq 3$, the graph $\kout{n}{r}{2}$ almost surely has a Hamiltonian Berge cycle.  Supposing $\mathcal{H}$ is selected from $\kout{n}{r}{2}$, we construct a random directed graph from $\mathcal{H}$ as follows.  For each $v$, we randomly pick one edge of $E_v$ and label it $e^{-} _v$, and we label the other edge $e^{+} _v$.  We then draw a directed arc from $u$ to $v$ for each $u \in e^{-} _v \setminus \{v\}$ and we draw a directed arc from $v$ to $w$ for each $w \in e^{+} _v \setminus \{v\}$.  Let $D$ be the directed graph obtained in this way.

The construction of $D$ has the same distribution as the process where for each $v$ we select $r-1$ `out' neighbors of $v$ and $r-1$ `in' neighbors of $v$.  This process results in the $(r-1)$-in, $(r-1)$-out random directed graph.\footnote{As is usual, the question of whether or not these neighbors are selected with replacement is irrelevant for the results we will cite.}  For this model, Cooper and Frieze \cite{digraph} proved that for each $k \geq 2$ the $k$-in, $k$-out directed graph is almost surely Hamiltonian.  Thus, 
there is almost surely an ordering of the vertices $v_1, \ldots, v_n$ such that $(v_i, v_{i+1})$ is an arc of $D$ for all $i$ (with indices viewed modulo $n$).

Each arc $(u, v)$ of $D$ corresponds to either $e_{u} ^+$ or $e_v ^{-}$ in $\mathcal{H}$, so if we chose such an edge of $\mathcal{H}$ for each arc $(v_1, v_{2}), (v_2, v_3) , \ldots , (v_{n}, v_{1})$, we cannot possibly choose the same edge twice unless there are two distinct indices such that $e_{v} ^{\pm} = e_{u} ^{\pm}$.  But by Lemma \ref{distinct edges}, almost surely all of these edges are distinct.  Thus, this directed Hamiltonian cycle in $D$ almost surely corresponds to a Hamiltonian Berge cycle in $\mathcal{H}$, as desired.

\hrulefill

On the other hand, we claim that $\kout{n}{r}{1}$ almost surely has vertices contained in only one edge (which would imply it is not Berge Hamiltonian).  For each $v \in V$, let $e_v$ denote the edge chosen by vertex $v$, and let $\tilde{e}_v$ be an $r$-set of $V$ chosen uniformly at random from among all $r$-sets containing $e_v \setminus \{v\}$.  Consider $U = \bigcup_{v \in V} (e_v \setminus \{v\})$ and $\tilde{U} = \bigcup _{v \in V} \tilde{e}_v$.  The set $U$ is equal to the vertices of $\kout{n}{r}{1}$ appearing in more than one edge, and by construction $U \subseteq \tilde{U} \subseteq V$.  Moreover, $\tilde{U}$ consists of the union of $n$ independent $r$-sets each chosen uniformly at random.\footnote{The set $\tilde{U}$ is introduced for a technical simplification arising from the fact that the sets $e_v \setminus \{v\}$ are not identically distributed, but the sets $\tilde{e}_v$ are.}  Thus, we have
\[
\mathbb{P} \left( \text{all vertices of $\kout{n}{r}{1}$ appear in at least $2$ edges} \right) = \mathbb{P}(U = V) \leq \mathbb{P}(\tilde{U} = V).
\]
The event $\tilde{U} = V$ is a special case of the famous \textit{generalized coupon collector problem}, where each edge corresponds to a ``coupon," each ``coupon" consists of $r$ symbols, and $\tilde{U}= V$ is the event that the union of $n$ random coupons contains all $|V|$ symbols (here, $|V| =n$).  This has been well-studied, and $\mathbb{P}(\tilde{U} = V)$ tends to $0$ since $r n \ll |V| \log |V|$ (see, e.g., \cite{erdosCoupon}).  Therefore, $\kout{n}{r}{1}$ almost surely has vertices contained in only one edge, and thus almost surely it is not Berge Hamiltonian.
\end{proof}

Finally, we handle the case of weak Berge Hamiltonicity in $k$-out $r$-graphs.
\begin{theorem}\label{thm: k-out-weak}
For any fixed $r\ge 4$, $\kout{n}{r}{1}$ almost surely has a weak Hamiltonian Berge cycle, whereas $\kout{n}{3}{1}$ almost surely does not.
\end{theorem}

\begin{proof}
[Proof of Theorem \ref{thm: k-out-weak}]
$\kout{n}{3}{1}$ almost surely contains three distinct vertices of degree 1 which all share a common neighbor. Thus this graph is almost surely not Hamiltonian.

On the other hand, we can embed an $(r-1)$-out graph in the $1$-out $r$-graph.  Namely, each vertex $x$ picks a hyper edge $S_x$, and we then include in our graph every edge of the form  $xy$  for $y$ in $S_x$ .
 This gives us an $(r-1)$-out graph, which has a Hamiltonian cycle when $r\ge 4$ (see \cite{bohman2009}).  A Hamiltonian cycle in this graph is a weak Hamiltonian Berge cycle in our hypergraph.

\end{proof}

\bibliographystyle{plain}
\bibliography{myBib.bib}

\appendix
\section{Proof of Lemma \ref{stopping contains boosters}}
\begin{proof}[Proof of Lemma \ref{stopping contains boosters}]
Let $G(m)$ be the random hypergraph with $m$ edges, and let $M$ be the stopping time for when $G(m)$ has minimum degree at least 2.  With high probability $M \in [m_1, m_2]$ where $m_1 = (n/r) \log(n)/2$ and $m_2 = 2(n/r) \log(n)$. 

Let $N = {n \choose r}$ and $\gamma = \varepsilon n\log n + n$.  Then by Lemma \ref{boosters in expanders}, any $(n/4, 2)$-expander has at least $N c_r'$ boosters (for some constant $c_r'$).  A union bound over $M$ and over the choice of $\Gamma$ (and $|\Gamma|$) gives the chance that $G$ contains some $\Gamma$ but none of its boosters is at most
\begin{eqnarray*}
\mathbb{P}(bad) - o(1) &\leq& \sum_{m= m_1} ^{m_2} \sum_{i \leq \gamma} \dfrac{{N \choose i} {N-i - c_r' N \choose m-i}}{{N \choose m}} \leq \sum_{m= m_1} ^{m_2} \sum_{i \leq \gamma} \exp \left[ \dfrac{-c_r' N (m-i)}{N-i}\right] \dfrac{{N \choose i} {N-i \choose m-i}}{{N \choose m}}\\
&\leq& \sum_{m= m_1} ^{m_2} \sum_{i \leq \gamma} \exp \left[ -c_r' m/100 \right ] {N \choose i} \left(\dfrac{m}{N}\right)^i \leq \sum_{m= m_1} ^{m_2} \sum_{i \leq \gamma} \exp \left[ -c_r' m/100 \right ] \left( \dfrac{e m}{i} \right)^{i}\\
&\leq& \sum_{m= m_1} ^{m_2} \gamma \exp \left[ -c_r' m/100 \right ] \left( \dfrac{e m}{\gamma} \right)^{\gamma} = o(1)
\end{eqnarray*}
(with the initial $o(1)$ corresponding to $M \notin [m_1, m_2]$).
\end{proof}

\section{Proof of Lemma \ref{random has nice conditions}}

\begin{proof}[Proof of Lemma \ref{random has nice conditions}]
Each piece is straightforward and only involves tail bounds for binomial coefficients. 
In fact (P2) and (P3) are already proven in \cite{kahn} (Lemma 5.1(c)).  
We will need the Chernoff bound.

\textbf{Chernoff:}   Say $X \sim Bin(N,p)$ and $\phi(x) = (1+x)\log(1+x)-x$.  Say $\mu = Np$ and $t \geq 0$.  Then we have
\begin{equation}\label{chernoff}
\mathbb{P}(X \geq \mu + t) \leq \exp \left[ -\mu \phi(t/\mu) \right].
\end{equation}

\textbf{(P1)} Let $(u, S)$ be a vertex $u \in [n]$ and a set of edges $S$ of size $|S| = t$ such that each contains $u$.  Then the expected number, $X$, of pairs $(u,S)$ of this form where $S \subseteq E(G)$ is
\[
\mathbb{E}[X] = n {{n-1 \choose r-1} \choose t} p^t \leq n \left[ \dfrac{{n-1 \choose r-1}e}{t} \right]^t  p^t = n \left[ \dfrac{{n-1 \choose r-1}e}{t} \right]^t  \left(c_r \dfrac{\log(n)}{{n-1 \choose r-1}} \right)^t = n \left[ \dfrac{e c_r \log(n)}{t} \right]^t,
\]
where in fact $c_r \in (1/2, 2)$ [since the stopping time is definitely in there].  But all the same, we're done by picking $t = C_r \log(n)$ for some sufficiently large constant $C_r$ (in fact $C_r = 10$ would suffice]).

\textbf{(P2) and (P3)}  are both proven in \cite{kahn}.  

\textbf{(P4)} Let $U$ be fixed and $|U|=u$.  Then \#(edges hitting $U$ twice) is stochastically dominated by a binomial with $p = c_r \log(n) / n^{r-1}$ and $N=c_r ' |U|^2 n^{r-2}$ 
So $\mu = Np = C_r \log(n) u^2 / n$, and set $t = u \log(n)^{3/4}/2$.  Then (using $\mu = o(t)$)
\[
\mathbb{P}(\# \geq 2t) \leq \mathbb{P}(\# \geq \mu + t) \leq \exp \left[ -\dfrac{C_r \log(n) u^2}{n} \left( \dfrac{n}{2.1C_r u \log(n) ^{1/4}} \log(n / (2C_r u \log(n) ^{1/4})) \right) \right].
\]
Taking a union bound over $U$ with $|U| = u$ and summing over $u$ gives
\begin{eqnarray*}
\mathbb{P}(\text{not (P4)}) &\leq& \sum_{u} {n \choose u} \exp \left[ -\dfrac{C_r \log(n) u^2}{n} \left( \dfrac{n}{2.1 C_r u \log(n) ^{1/4}} \log(n / (2C_r u \log(n) ^{1/4})) \right) \right]\\
&\leq& \sum_{u} ( en/u)^u \exp \left[ -\dfrac{\log(n)^{3/4} u}{2.1} \log(n / (2 C_r u \log(n) ^{1/4})) \right]\\
&\leq& \sum_{u}\exp \left[u \log(en/u) -\dfrac{\log(n)^{3/4} u}{2.1} \log(n / (2 C_r u \log(n) ^{1/4})) \right]\\
&\leq& \sum_{u} \exp \left[ \dfrac{-u \log(n)^{3/4}}{20} \log \log(n) \right] = \mathcal{O}\left( \exp \left[ \dfrac{- \log(n)^{3/4}}{20} \log \log(n) \right] \right) = o(1).
\end{eqnarray*}
where the last line holds by summing the geometric series.  From the second to last line to the last, we lose some constant (absorbed in the `1/20') to take care of the $\log(en/u)$ term (and others).
\ignore{\begin{eqnarray*}
&=& \log(n/u) -\dfrac{\log(n)^{3/4}}{2.1} \log(n / ( u \log(n) ^{1/4})) = \log(n/u) \left[ 1 -\dfrac{\log(n)^{3/4}}{2.1} \right] + \dfrac{\log(n)^{3/4}}{2.1 \cdot 4}\log \log(n)\\
&\leq& \log(n/u) \left[-\dfrac{\log(n)^{3/4}}{2.1 \cdot (1\pm o(1))} \right] + \dfrac{\log(n)^{3/4}}{2.1 \cdot 4}\log \log(n)
\end{eqnarray*}
}

\textbf{(P5)} For every pair of disjoint vertex sets $U,W$ of sizes $|U| \leq \dfrac{n}{\log(n)^{1/2}}$ and $|W| \leq |U| \log(n) ^{1/4}$, there are at most $\dfrac{\varepsilon \log(n) |U|}{2}$ edges of $G$ meeting $U$ exactly once and also meeting $W$.  Say $U$ and $W$ are fixed and $|U| = u$ and (wlog) $|W|=u \log(n)^{1/4}$.  The number of edges meeting $U$ exactly once and also $W$ is bounded above by a binomial with parameters $p = c_r \log(n) / n^{r-1}$ and $N = |U| |W| n^{r-2}$.  So $\mu = Np = c_r u^2 \log(n)^{5/4} /n$, and set $t = \varepsilon u \log(n) / 4$.  Again using $\mu = o(t)$, and taking a union bound over choices of $U$ and $W$, we have

\begin{eqnarray*}
\mathbb{P}(\text{not (P5)}) &\leq& \sum_{u =1} ^{n/\sqrt{\log(n)}} {n \choose u} {n \choose u \log(n)^{1/4}} \mathbb{P}(\# \geq 2t) \leq \sum_{u =1} ^{n/\sqrt{\log(n)}} {n \choose u \log(n)^{1/4}}^2 \mathbb{P}(\# \geq \mu + t)\\
&\leq& \sum_{u =1} ^{n/\sqrt{\log(n)}} \exp \left[2u \log(n)^{1/4} \log\left( \dfrac{n}{u \log(n)^{1/4}} \right) - \dfrac{\varepsilon u \log(n)}{4.1} \log\left( \dfrac{n}{u \log(n)^{1/4}} \right) \right]\\
&\leq& \sum_{u =1} ^{n/\sqrt{\log(n)}} \exp \left[- \dfrac{\varepsilon u \log(n)}{4.2} \log\left( \dfrac{n}{u \log(n)^{1/4}} \right) \right] \leq \sum_{u =1} ^{n/\sqrt{\log(n)}} \exp \left[- \dfrac{\varepsilon u \log(n)}{20} \log \log(n) \right]\\
&\leq & \sum_{u =1} ^{\infty} \exp \left[- \dfrac{\varepsilon u \log(n)}{20} \log \log(n) \right] = \mathcal{O} \left(\exp \left[- \dfrac{\varepsilon \log(n)}{20} \log \log(n) \right]\right) = o(1).
\end{eqnarray*}

\textbf{(P6)} Let $U$ have size $n/\log(n)^{1/2}$ and $W$ have size $n/4$.  Then the number of edges meeting $U$ exactly once and $W$ exactly $r-1$ times is a binomial random variable with $p = c_r \log(n) / n^{r-1}$ and $N= |U| |W|^{r-1} c'_r = c''_r n^{r} / \log(n)^{1/2}$.  So we have $\mu = Np = C_r n \sqrt{\log(n)}$ and set $t = n \log(n)^{1/3}$.  Now we need to use another part of Chernoff's bound that $P(X < \mu - t) \leq  \exp(- t^2 /(2\mu))$.

Then taking a union over $U$ and $W$, we have
\begin{eqnarray*}
\mathbb{P}(\text{not (P6)}) &\leq& {n \choose n/\sqrt{\log(n)}} {n \choose n/4 } \mathbb{P}(\# < t) \leq 4^{n} \mathbb{P}(\# \leq \mu - t) \leq 4^n \exp \left[ \dfrac{-t^2}{2\mu} \right]\\
&=& \exp \left[n \log(4) - \dfrac{n^2 \log(n) ^{2/3}}{2C_r n \log(n)^{1/2}} \right] = \exp \left[n \log(4) - \dfrac{n \log(n) ^{1/6}}{2C_r} \right] = o(1).
\end{eqnarray*}

\textbf{(P7)} Suppose we went from $G$ to $\Gamma_0$, and let's analyze the probabilty that there exists a pair $(U,W)$ violating (P7).  Then (since $G$ has (P3) and (P6) with high probability), the only thing that could have gone wrong is that every edge going across was missing.  This is bounded by
\begin{eqnarray*}
\mathbb{P}(\text{not (P7)}) &\leq& 4^{n} \prod_{u \in U} \left(\dfrac{{d_G (u) - e_G (u, W) \choose \varepsilon \log(n)}}{{d_G (u) \choose \varepsilon \log(n)}} \right) \leq 4^n \prod_{u \in U} \exp \left[ \dfrac{-\varepsilon \log(n) e_G (u,W)}{d_G (u)} \right]\\
&\leq& \exp \left[n \log(4) - \dfrac{\varepsilon \log(n) e_G (U,W)}{\Delta(G)} \right] \leq \exp \left[n \log(4) - \dfrac{\varepsilon \log(n) n \log(n)^{1/3}}{\Delta(G)} \right] = o(1),
\end{eqnarray*}
where we used (P1) to conclude $\Delta(G) = o(\log(n)^{4/3})$.
\end{proof}

\end{document}